\documentclass[a4paper,10pt]{article}
\usepackage{authblk}
\usepackage[utf8]{inputenc}
\usepackage[english]{babel}
 \usepackage{fullpage}

\usepackage[usenames,dvipsnames]{color}
\definecolor{grey}{rgb}{.7,.7,.7}
\definecolor{blue}{rgb}{0,0,.8}
\definecolor{red}{rgb}{.8,0,0}
\definecolor{green}{rgb}{0,.4,0}
\definecolor{gold}{rgb}{0.8,0.6,0.1}
\definecolor{brown}{rgb}{0.8,0.4,0.1}
\definecolor{pink}{rgb}{1,0.3,0.3}

\usepackage{enumitem}

\usepackage{amsmath, amsthm, amsfonts, amssymb}
\usepackage{graphicx}

\newcommand{\R}{\mathbb R}

\newcommand{\A}{\mathcal A}
\newcommand{\B}{\mathcal B}

\newcommand{\G}{\mathcal G}

\renewcommand{\L}{\mathcal L}

\DeclareMathOperator{\ind}{ind}
\DeclareMathOperator{\conv}{conv}

\DeclareMathOperator{\Img}{\mathbf{Im}}

\newcommand{\heading}[1]{\vspace{1ex}\par\noindent{\bf\boldmath #1}}

\newtheorem{theorem}{Theorem}
\newtheorem{observation}[theorem]{Observation}
\newtheorem{lemma}[theorem]{Lemma}
\newtheorem{prop}[theorem]{Proposition}

\newcounter{sideremark}

\title{Intersection patterns of planar sets\footnote{The project was supported by ERC Advanced Grant 320924. GK was also partially supported by NSF grant DMS1300120.
The research stay of ZP at IST Austria is funded by the project
CZ.02.2.69/0.0/0.0/17\_050/0008466 Improvement of internationalization in
the field of research and development at Charles University, through the
support of quality projects MSCA-IF.}}

\author[1]{Gil Kalai}
\author[2,3]{Zuzana Pat\'akov\'a}
\date{}

\affil[1]{\small Hebrew University of Jerusalem, Israel.}
\affil[2]{\small Computer Science Institute of Charles University, Prague, Czech Republic.}
\affil[3]{\small IST Austria, Klosterneuburg, Austria.}

\begin{document}
\maketitle{}
\abstract{
Let $\A=\{A_1,\ldots,A_n\}$ be a family of sets in the plane. For $0 \leq i < n$, denote by $f_i$ the number of subsets $\sigma$ of $\{1,\ldots,n\}$ of cardinality $i+1$ that satisfy $\bigcap_{i \in \sigma} A_i \neq \emptyset$. Let $k \geq 2$ be an integer. We prove that if each $k$-wise and $(k+1)$-wise intersection of sets from $\A$ is empty, or a single point, or both open and path-connected, then $f_{k+1}=0$ implies $f_k \leq cf_{k-1}$ for some positive constant $c$ depending only on $k$. Similarly, let $b \geq 2, k > 2b$ be integers. We prove that if each $k$-wise or $(k+1)$-wise intersection of sets from $\A$ has at most $b$ path-connected components, which all are open, then $f_{k+1}=0$ implies $f_k \leq cf_{k-1}$ for some positive constant $c$ depending only on $b$ and $k$.
These results also extend to two-dimensional compact surfaces.}

\section{Introduction}
Consider a finite collection of sets $\A = \{A_1,A_2,\dots,A_n\}$. Set $A_\sigma = \bigcap_{i \in \sigma} A_i$ for $\sigma \subseteq [n]$.
Let $N({\A})=\{\sigma \subseteq [n]: A_\sigma \ne \emptyset\}$ 
be the nerve of $\A$. We put \[
  f_k(\A)=f_k(N(\A))= |\{\sigma \in N(\A) \colon |\sigma|=k+1\}|.                            
                             \]

Helly's theorem asserts that for a finite collection of (at least $d+1$) \emph{convex} sets in $\R^d$, if every $d+1$ sets have a point in common then all the sets have a point in common. 
In other words, Helly's theorem asserts that if $f_{n-1}(\A)=0$ then $f_d(\A) < \binom{n}{d+1}$ provided $\A$ is a finite family of convex sets.   

A far-reaching extension of Helly's theorem was conjectured by Katchalski and Perles and proved by Kalai and by Eckhoff. 

\begin{theorem}[Kalai \cite{Kalai_upper_bound}, Eckhoff \cite{Eckhoff_upper_bound}]\label{t:upper_bound}
 Let $\A$ be a family of $n$ convex sets in $\R^d$, and suppose that each $d+r+1$ members of $\A$ have empty intersection.
 Then, for $k=d,\ldots, d+r-1,$
 \[
  f_k(\A) \leq \sum_{j=k}^{d+r-1}\binom{j-d}{k-d}\binom{n-j+d-1}{d}.
 \]
\end{theorem}

This ``upper bound theorem'' provides sharp upper bounds for 
$f_d(\A),\ldots,f_{d+r-1}(\A)$ in terms of $f_0(\A)$ provided $f_{d+r}(\A)=0.$ 
It implies sharp version of the ``fractional Helly theorem'' of Katchalski and Liu \cite{KL_frac-Helly}. Moreover, the bound cannot be improved as there is a simple case of equality: The family consists of $r$ copies of $\R^d$ and $n-r$ 
hyperplanes in general position.

We now move from families of convex sets to families of sets satisfying certain topological conditions. 
Helly himself found a topological extension to his theorem and finding topological versions to 
Helly-type theorems is a very interesting and fruitful area. 
For example, Goaoc, Pat\'ak, Pat\'akov\'a, Tancer, and Wagner \cite{GPPTW} found a 
far-reaching topological extension of Helly's theorem.

Recall that a family of open sets is a {\it good cover} if every non-empty intersection of sets in the family is null-homotopic.
(This condition is satisfied automatically for open convex sets since the intersection of convex sets is convex.)
The celebrated nerve lemma asserts that the union of the sets in a good cover is homotopically equivalent to the nerve.
Theorem \ref{t:upper_bound} extends to good covers in $\R^d$, and, as a matter of fact, it is enough to assume that all non-empty
intersections are homologically trivial. In this case, the nerve $N$ of the family has the $d$-Leray property:
For every induced subcomplex $N'$ of $N$  all reduced homology groups $\tilde H_i(N')$ vanish when $i \ge d$. See \cite {Kal02} for
how Theorem \ref{t:upper_bound} is derived from properties of either exterior or symmetric algebraic shifting and can also be deduced from Stanley's work
\cite {Sta75} on $f$-vectors of Cohen-Macaulay ring.  

It is believed that similar but weaker upper bounds to those in Theorem \ref{t:upper_bound}, 
apply under much weaker topological conditions than the condition of being a good cover.
Some conjectures in this direction were offered by Kalai and Meshulam, see \cite{Kalai_OWR}. 
The purpose of this paper is to prove such upper bounds for planar sets.

The main result of this paper is for nerves of planar sets. We give, under fairly weak conditions on the sets, 
 upper bound on $f_k(\A)$ in terms of $f_{k-1}(\A)$ under the assumption that $f_{k+1}(\A)=0$ (for details, see Theorems \ref{t:main_b=1} and \ref{t:main_b}).

\section{New results}\label{s:results}

Let $\A=\{A_1,\ldots, A_n\}$ be a finite family of sets.
We put 
\[
 f^{\ind}_k(\A) = |\{ \tau \in N(\A) \colon |\tau|= k+1 \text{ and } \exists \sigma \in N(\A), \ \sigma \supset \tau, |\sigma| = k+2 \}|.
\] 
In words, $f^{\ind}_k(\A)$ counts the number of intersecting $(k+1)$-tuples in $\A$ induced by all intersecting $(k+2)$-tuples in $\A$. Hence, $f^{\ind}_k(\A) \leq f_k(\A) \leq \binom{n}{k+1}$.

\begin{theorem}\label{t:main_b=1}
 Let $k \geq 2$ be an integer. Let $\mathcal A = \{A_1,\ldots,A_n\}$ be subsets of $\R^2$ satisfying the following conditions: 
\begin{enumerate}[label=(\roman{*})]
  \item  $f_{k+1}(\A) = 0$, \label{no_k+2_intersect}
  \item for each $\sigma$, $|\sigma| \in \{k,k+1\}$, $A_\sigma$ is empty, or a single point, or both open and path-connected.
  \label{path_conn_k,k-1}.
 \end{enumerate}
 Then $f_k(\A) \leq cf^{\ind}_{k-1}(\A)$, where $c=c(k)>0$ is a constant. 
 Specifically,  $c(k)= \frac{4k+1}{4k^2-7k+1}$.
\end{theorem}

If we allow more path-connected components of $A_\sigma$, we have the following result.

\begin{theorem}\label{t:main_b}
 Let $b\geq 2, k > 2b$ be integers. Let $\mathcal A = \{A_1,\ldots,A_n\}$ be subsets of $\R^2$ satisfying the following conditions: 
\begin{enumerate}[label=(\roman{*})]
  \item  $f_{k+1}(\A) = 0$, \label{no_k+2_intersect_b}
  \item for each $\sigma$, $|\sigma| \in \{k,k+1\}$, $A_\sigma$ is open\footnote{We note that empty set is, by definition, open.} and has at most $b$ path-connected components\label{path_conn_k,k-1_b}.
 \end{enumerate}
 Then $f_k(\A) \leq cf^{\ind}_{k-1}(\A)$, where $c=c(b,k)>0$ is a constant. 
 Specifically,  $c(b,k)= \frac{b}{k-2b}$.
\end{theorem}

Note that in contrast to Theorem \ref{t:main_b=1} we assume that each path-connected component of $A_\sigma$ is open. In particular, we don't allow components consisting of a single point. Also, although we can prove the bound on $f_k(\A)$ only for $k > 2b$, we believe that a similar bound is true for any $k \geq 2$.

\smallskip
 We note that the plane in Theorems \ref{t:main_b=1} and \ref{t:main_b} can be replaced by any surface. By a \emph{surface} we mean a two-dimensional compact real manifold. Indeed, some parts of the proof are independent of the surface (Observations \ref{o:simple_b=1} and \ref{o:simple}, Proposition \ref{l:planar_embedding_b=1}). The remaining parts require a surface analogue of Euler's formula for planar graphs $v-e+f \geq 2$ (Observation \ref{o:euler_char}), and more careful analysis depending on the sign of Euler characteristic of the surface. More details are provided in Section \ref{s:manifold}.

\begin{theorem}\label{t:manifold}
 Let $M$ be a surface and let $\mathcal A = \{A_1,\ldots,A_n\}$ be a family of its subsets.

Let $k\geq 2$ be an integer and let  $\mathcal A$ satisfies the conditions \ref{no_k+2_intersect} and \ref{path_conn_k,k-1} of Theorem \ref{t:main_b=1}.
 Then $f_k(\A) \leq c_1f^{\ind}_{k-1}(\A) + c_2$, where $c_1 > 0$, $c_2 \geq 0$ are constants depending only on $k$ and the surface~$M$.
 
 Similarly, let $b \geq 2, k > 2b$ be integers and let $\A$ satisfies the conditions \ref{no_k+2_intersect_b} and \ref{path_conn_k,k-1_b} of Theorem~\ref{t:main_b}.
  Then $f_k(\A) \leq c_1f^{\ind}_{k-1}(\A) + c_2$, where $c_1 > 0$, $c_2 \geq 0$ are constants depending only on $k, b$ and the surface~$M$.
\end{theorem}

Let us assume for a moment that $\A$ is a family of $n$ open sets in the plane.
When $b=1$ the condition on the set system $\A$ is that all $k$-wise and $(k+1)$-wise intersections are either path-connected or empty. In this case Theorem \ref{t:main_b=1} asserts that if $f_3(\A) = 0$, then $f_2(\A) \leq 3 f_1(\A)$, and, if $f_{k+1}(\A)=0$ for $k \geq 3$, then $f_k(\A) \leq \frac{4k+1}{4k^2-7k+1}f_{k-1}(\A)$.
Theorem \ref{t:upper_bound} asserts that when all sets are convex, then
\begin{equation}\label{e:kalai}
  f_3(\A) = 0 \quad \Rightarrow \quad f_2(\A) \leq \binom{n-1}{2},
\end{equation}
and this inequality continues to hold if all intersections are either contractible or empty.
We show that the weaker condition that all intersections are path-connected or empty does not suffice for (\ref{e:kalai}).

\begin{theorem}\label{t:construction}

 For any $n \geq 6$, there is a family $\A$ of $n$ open sets in $\R^2$ such that intersection of every subfamily is either empty or path-connected, and for which
\[
 f_3(\A) = 0 \quad \text{ and }  \quad f_2(\A) \geq \binom{n-1}{2}+\frac13(n^2-6n+3).
\]
For infinitely many $n$ there exists a family $\mathcal B$ of $n$ open sets in $\R^2$ satisfying the same condition and for which $f_3(\B)=0$ and $f_2(\B) = \binom{n-1}{2}+\frac13(n^2-4n+3).$

Furthermore, in both cases all but one of the sets of $\A$ and $\B$, respectively, are contractible and intersections of all pairs and triples of sets are contractible.
\end{theorem}

\heading{Discussion on the sharpness of the bounds}.
In general, the bounds in Theorems \ref{t:main_b=1}, \ref{t:main_b} and \ref{t:manifold} are not sharp and we do not pose any conjecture on the best values of the constants. The main loss appears in the proof of Lemma \ref{l:planar_edge_upper_bound}, partly caused by applying Observation \ref{o:planar} to graphs with very specific triangular structure. In the proofs of Theorems \ref{t:main_b=1} and \ref{t:main_b} we also neglect additive constants, which we, in general, cannot afford in Theorem \ref{t:manifold}. Furthermore, it remains open whether the construction in Theorem \ref{t:construction} can be improved or not.

\medskip
\begin{flushleft}
\begin{minipage}{.68\textwidth}
We note that for $n=4, b=1$ our proof of Theorem \ref{t:main_b=1} in fact gives $f_2(\A) \leq 1/3f_1^{\ind}(\A) + 2$, which is optimal. Indeed, we construct a family $\A=\{A_1,A_2,A_3,A_4\}$ of closed sets with $f_3(\A)=0$ and $f_2(\A)=4$.  Let $a_i$, $1\leq i \leq 3$, be three points in the plane and let $a_4$ be the barycenter of the triangle $a_1a_2a_3$. Let $A_1 = \conv(a_2,a_3,a_4), A_2= \conv(a_1,a_3,a_4), A_3= \conv(a_1,a_2,a_4)$ and let $A_4$ be the boundary of the triangle $a_1a_2a_3$. It is easy to see that every three sets intersect, but not all of them.  Taking an $\varepsilon$-neighborhood of $A_i$, $1\leq i \leq 3$, for a sufficiently small $\varepsilon$, with obtain a family of open sets with the same nerve as $\A$. 

Theorem \ref{t:construction} further develops this construction, for details we refer to Section~\ref{s:construction}.
\end{minipage}
\begin{minipage}{.2\textwidth}
\begin{flushright}
  \includegraphics[page=6,keepaspectratio=true]{f_stellar}
\end{flushright}
\end{minipage}
\end{flushleft}

\heading{Organization of the paper: }
In Section \ref{s:b=1} we prove Theorem \ref{t:main_b=1}.
The proof of Theorem \ref{t:main_b} is given  
in Section \ref{s:proof_t:main}, the adaptation to the surface setting is in Section \ref{s:manifold},
 and the proof of Theorem~\ref{t:construction} is in Section~\ref{s:construction}.

\section{Path-connected intersections}\label{s:b=1}

 We start with a proof of Theorem \ref{t:main_b=1}. 

Recall that $A_\sigma = \bigcap_{i \in \sigma} A_i$, where $\sigma \subseteq [n]$. 
By assumptions, for any $\sigma$ with $|\sigma| \in \{k,k+1\}$, a nonempty $A_\sigma$ is path-connected and if it is not a single point, it is also open.
We want to get an upper bound on the number of intersecting $(k+1)$-tuples of $\A$ provided no $k+2$ sets from $\A$ intersect.

Let $N$ denote the nerve of $\A$ and let $\sigma \in N$ be a face of size $k+1$.
For every nonempty $A_\sigma$, choose $a_\sigma \in A_\sigma$. Let $W$ be the set of all $a_\sigma$'s, that is 
\[
 W=\{a_\sigma\colon \sigma \in N, |\sigma| = k+1\}.
\]
By assumption, no $k+2$ sets of $\A$ intersect, so $a_\sigma \neq a_{\sigma'}$ if and only if $\sigma \neq \sigma'$. Hence $|W| = f_k(N(\A))=f_k(\A)$.

\heading{Embeddability.}
The aim now is to construct a planar graph $G$ on the vertices $W$ which can be drawn properly inside $\bigcup \A$.  We double-count the number of edges of $G$ in order to get the desired bound on the number of vertices of $G$, and hence on $f_k(\A)$.

\smallskip
Let $\tau \in N$ be a face of size $k$ which is contained in at least one $(k+1)$-element face of $N$. Set
\[
 V_\tau = \{a_\sigma \colon \sigma \in N, |\sigma|=k+1, \tau \subset \sigma\} \qquad \text{and} \qquad \Gamma = \{\tau \colon V_\tau \neq \emptyset\}.
\]

We have $W = \bigcup_{\tau \in N} V_\tau$ and $|\Gamma| = f^{\ind}_{k-1}(\A)$.

Let $G_\tau$ be a tree on the vertex set $V_\tau$, where $|V_\tau| \geq 2$. Let us describe the embedding $\G_\tau\colon G_\tau \to A_\tau$. Let $a_\sigma, a_{\sigma'} \in V_\tau$ be two vertices of $G_\tau$ connected by an edge. By assumptions, $A_\tau$ is path-connected, hence the edge between $a_\sigma$ and $a_{\sigma'}$ can be drawn solely inside $A_\tau$.
Moreover, for $|V_\tau| \geq 2$,  $A_\tau$ is open, so we can assume  the tree $G_\tau$ is embedded into $A_\tau$ piece-wise linearly.

\begin{observation}\label{o:simple_b=1}
 For any choice of trees $G_\tau$ and their embeddings $\G_\tau$ described above, the following is true. Define a graph $G$ on the vertex set $W$ by considering a union of all trees $G_\tau$ with $|\tau|=k$. We also define a mapping 
 $\G \colon G \to \bigcup \A$ as $\G= \bigcup_{\tau \colon |\tau|=k}\G_\tau$.
 Then 
 \begin{enumerate}
  \item $G$ is a simple graph, that is, it has no loops and no multiple edges. \label{o:simple_i_b=1}
  \item $|E(G)| = (k+1)|W| - f^{\ind}_{k-1}(\A)$. In particular, $|E(G)| \geq (k+1)f_k(\A) - \binom{n}{k}$.\label{o:simple_ii_b=1}
 \item The embeddings $\G_\tau: G_\tau \to A_\tau$ can be modified so that
  under the mapping $\G$, the images of any two edges of $G$ intersect in a finite number of points, and for any edge $z$ of $G$, $\G(z) \cap \G(W)$ contains exactly the endpoints of $z$.\label{o:simple_iii_b=1}
 \end{enumerate}
\end{observation}

\begin{proof}

\begin{enumerate}
\item Immediately follows from the construction. Indeed, there are no loops since all $G_\tau$'s are trees. Moreover, any edge between $a_\sigma,a_{\sigma'}$ belongs to $G_{\sigma \cap \sigma'}$, hence every edge in $G$ appears with multiplicity one.

\item Observe that each vertex of $G$ belongs to exactly $k+1$ graphs $G_\tau$. Indeed, $a_\sigma$ belongs to those $G_\tau$ for which $\tau \subset \sigma$. Since $|\sigma| = k+1$ and $|\tau|=k$, there are exactly $k+1$ such graphs.

Recall $\Gamma = \{\tau \colon V_\tau \neq \emptyset\}$. 
Each $G_\tau$ is a tree, hence $|E(G_\tau)| = |V_\tau|-1$. Therefore,
\[
 |E(G)| = \sum_{\tau \in \Gamma} |E(G_\tau)| = \sum_{\tau \in \Gamma}{\left(|V_\tau|-1\right)} = (k+1)|W| - |\Gamma|.
\]
Using that $|\Gamma| = f^{\ind}_{k-1}(\A) \leq f_{k-1}(\A) \leq \binom{n}{k}$ and $|W| = f_k(\A)$, we get the desired bounds.

\item Since the embeddings of $G_\tau$'s are piece-wise linear, having infinitely many crossings in $\G$ means that some two segments coincide, say segments of two edges $e$ and $f$. In such case, we slightly perturb images of $e$ and $f$ keeping the endpoints fixed. Notice that we heavily use here that if $G_\tau$ contains an edge, the set $A_\tau$ is open and path-connected. 
The second part follows, since when embedding the edges of $G_\tau$ into an open set $A_\tau \subseteq \R^2$, we can always avoid any zero-dimensional set of vertices.\qedhere
\end{enumerate}  
\end{proof}

\heading{Redrawing.} The crucial part of the whole proof is to show that we can choose trees $G_\tau$ so that the graph $G$ is planar.
\begin{prop}\label{l:planar_embedding_b=1}
 There is a choice of trees $G_\tau$ on the vertex sets $V_\tau$ together
 with  embeddings $\G_\tau \colon G_\tau \to A_\tau$  such that:
 \begin{itemize}
 \item the graph $G$, defined on the vertex set $W$ as $G = \bigcup_{\tau \colon |\tau| = k} G_\tau$, is a simple planar graph, 
 \item the union of embeddings $\G_\tau$ provides an embedding of $G$ into $\bigcup \A$.
 \end{itemize}
\end{prop}

\begin{proof}
Let $ \mathcal G_\tau \colon G_\tau \to A_\tau$ where $|\tau|=k$, be embeddings as described above and let
$\mathcal G := \cup_{\tau: |\tau|=k} \mathcal G_\tau$ be a mapping.
By Observation \ref{o:simple_b=1}(\ref{o:simple_i_b=1}),  the graph $G = \bigcup_{\tau \colon |\tau| = k} G_\tau$ is simple. All we need to do is get rid of all intersection points in $\Img(\G)$ while preserving that, for each $\tau$, $G_\tau$ is a tree (on the vertex set $V_\tau$) piece-wise linearly embedded into $A_\tau$.

 Let $e_\tau \in E(G_\tau), e_{\nu}\in E(\G_{\nu})$ be two edges that \emph{intersect}, that is $\G(e_\tau) \cap \G(e_{\nu}) \neq \emptyset$
  and the preimages of the \emph{intersection points} are not vertices of $e_\tau, e_{\nu}$. We now show how to change  $\mathcal G$ into an embedding  $\mathcal G'$. All
 changes will be done on the level of trees $G_\tau$ and their embeddings $\mathcal G_\tau$.
By this procedure, we remove all intersection points in $\Img(\mathcal G)$. 
By Observation \ref{o:simple_b=1}(\ref{o:simple_iii_b=1}), there is just a finite number of intersection points, so we can deal with them one by one.
  Denote the vertices of $e_\tau$ by $x,y$ and the vertices of $e_\nu$ by $u,v$.
  
  Abusing notation, we will call the images 
of $u,v,x,y$ under $\mathcal G$ still $u,v,x,y$.
Fix an intersection point $p \in \mathcal G(e_\tau) \cap \mathcal G(e_{\nu})$ such that 
$p \notin \{u,v,x,y\}$.

By assumption, $\mathcal G_\tau$ and $\mathcal G_{\nu}$ are embeddings, hence $\tau \neq \nu$, and by condition \ref{no_k+2_intersect} (Theorem~\ref{t:main_b=1}), 
$|\tau \cup \nu|=k+1$ since $|\tau| = |\nu| = k$ and $f_{k+1}(\A)=0$. 
Since $\mathcal G(e_\tau) \subseteq \Img \mathcal G_\tau \subseteq A_\tau$ and 
$\mathcal G(e_{\nu}) \subseteq \Img \mathcal G_{\nu} \subseteq A_{\nu}$, it follows that 
$p \subseteq A_\tau \cap A_{\nu} = A_{\tau \cup \nu}$. Putting $\sigma = \tau \cup \nu$, we see that  $p \in A_\sigma$.
Since $|\sigma|=k+1$, there is a path $s$ in $A_\sigma$ between $p$ and $a_\sigma$ (condition \ref{path_conn_k,k-1} of Theorem~\ref{t:main_b=1}).

It can happen that some images of edges of $G$ intersect the path $s$ (see Figure \ref{f:intersection_s} left).
 In such case, we order the intersection points along $s$ and choose an edge $e_\mu$ of $G_\mu$ whose image $\G(e_\mu)$ intersects $s$ closest to the point $a_\sigma$.  By condition \ref{no_k+2_intersect}, $\mu \subset \sigma$, hence $a_\sigma \in G_\mu$. 
Denote by $g,h$ the two endpoints of $e_\mu$. Since $G_\mu$ is a tree, there is a path from $a_\sigma$ to exactly one of the vertices $g$ and $h$ in $G_\mu$.
Suppose that the first possibility occurs.
Then we replace $e_\mu \in E(G_\mu)$ by the edge connecting $h$ and $a_\sigma$, and modify the corresponding embedding $\mathcal G_\mu$:
we start at $h$, go along $e_\mu$ towards $s$, stop shortly before hitting $s$ and then continue along $s$ to $a_\sigma$ (we use that $A_\sigma$ is open), see Figure \ref{f:intersection_s}.
Note that after this change, $G_\mu$ is still a tree, hence it has the same number of edges as before. Also, we have not introduced any new crossings. We repeat this procedure until $\Img \G \cap s = \{p\}$.
\begin{figure}
\begin{center}
 \includegraphics[page=7]{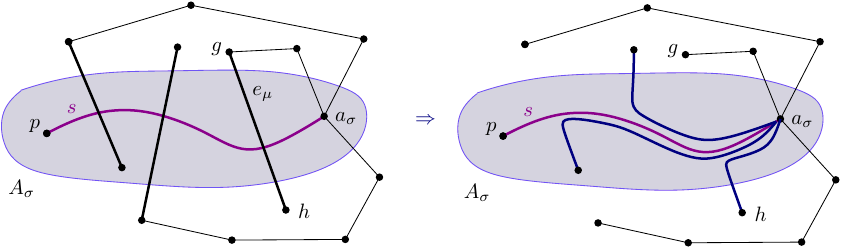}
 \caption{Removing several intersections of $\Img \mathcal G$ and the path $s$.}\label{f:intersection_s}
\end{center}
\end{figure}

Now we can describe how to get rid of the intersection point $p$.
Let us construct an auxiliary (multi)graph $H_\tau$ from $G_\tau$ by subdividing
the edge $xy$ at the \emph{new} vertex $p$ and adding an edge $a_\sigma p$. Formally,
\begin{eqnarray*}
   V(H_\tau) &:=& W_\tau \cup \{p\} \\
   E(H_\tau) &:=& (E(G_\tau) \setminus \{xy\}) \cup \{xp, yp, a_\sigma p\}.
\end{eqnarray*} 

Note that $H_\tau$ is \emph{almost a tree}, by removing the edge $a_\sigma p$ from $H_\tau$ we get a tree. Note also that if one of the vertices $x,y$ equals to $a_\sigma$, then $H_\tau$ is a multigraph. 
Denote by $\mathcal H_\tau$ the corresponding embedding of $H_\tau$ into $A_\tau$, where $\mathcal H_\tau(a_\sigma p) = s$ and $\mathcal H_\tau(xp)$ is a restriction of $\G_\tau(xy)$, similarly for $\mathcal H_\tau(yp)$.
It follows from the construction that exactly one of the edges $xp, yp$ is contained in the unique cycle of $H_\tau$ (the cycle containing the edge $a_\sigma p$). Without loss of generality, let it be the edge $xp$.
We redefine $G_\tau$ to $G_\tau'$ as follows (see Figure \ref{f:crossings_b=1}): We replace the edge $e_\tau$ in $G_\tau$ 
by the edge $y a_\sigma$. To describe the embedding $\mathcal G'_\tau \colon G'_\tau \to A_\tau$,
it is enough to describe the embedding of the edge $y a_\sigma$:
it is a concatenation of $\mathcal H_\tau(y p)$ and the path $s$.

Analogically, we obtain $G_{\nu}'$ from $G_\nu$.  Let us describe its embedding $\G'_\nu$, in particular, the embedding of the edge $va_\sigma$. Let $U \subseteq A_\sigma$ be a small open neighborhood of the point $p$ such that $|U \cap \G_\nu(uv) \cap H_\tau(ya_\sigma)| \leq 2$. Note that $U \cap \G_\nu(uv) \cap H_\tau(ya_\sigma)$ always contain the point $p$.
To embed the edge $va_\sigma$, we go from $v$ to $p$ along the original embedding of 
$e_\nu$.  If $\mathcal H_\tau(y a_\sigma)$ intersects $\mathcal H_\nu(v p)$ in $U$, we stop shortly before hitting $\mathcal H_\tau(y a_\sigma)$, otherwise we stop shortly before reaching the point $p$. Then we continue along $s$ (sufficiently close) to $a_\sigma$.
Here we use that if the intersection of $(k+1)$-tuples contains at least two points, it is open (condition \ref{path_conn_k,k-1} in Theorem \ref{t:main_b=1}).

\begin{figure}
\begin{center}
 \includegraphics[page=3]{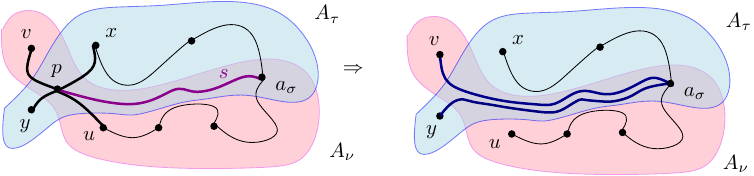}
 \caption{Left: $\Img \mathcal G$, right: $\Img \mathcal G'$}\label{f:crossings_b=1}
\end{center}
\end{figure}

Clearly, we removed at least one intersection point of $\mathcal G(e_\tau)$ and $\mathcal G(e_{\nu})$ and we didn't introduce any new crossing. 
The redefined $G_\tau'$, $G_{\nu}'$  are trees on the same set of vertices as before.
Repeatedly removing all intersection points of $\G$, we obtain an embedding $\mathcal G'$ of $G'  = \bigcup_{\tau \colon |\tau| = k} G'_\tau$. Since we didn't introduce any loops or multiple edges, by Observation \ref{o:simple_b=1}(\ref{o:simple_i_b=1}), we conclude that $G'$ is a simple planar graph.
\end{proof}

\heading{Proof of Theorem \ref{t:main_b=1} for $k \geq 3$.}
 By Observation \ref{o:simple_b=1}(\ref{o:simple_ii_b=1}),  
\[
 |E(G)| \geq   (k+1)|W| - f^{\ind}_{k-1}(\A).
\]

Since $G$ is a simple planar graph (Proposition \ref{l:planar_embedding_b=1}) on $|W|$ vertices, it has at most $3|W|$ edges. 
Combining it with $|W| = f_k(\A)$ we get a slightly worse bound:

\[
 f_k(\A) \leq \frac{1}{k-2}f^{\ind}_{k-1}(\A). 
\]

We derive the bound stated in Theorem \ref{t:main_b=1}  in the following subsection.
In order to make the approach above work for $k=2$, we need that the number of edges of $G$ is at most $c|W|$ for $c <3$.
This is indeed the case as shown in the next lemma.

\heading{The full proof of Theorem \ref{t:main_b=1}.}

 \begin{lemma}\label{l:edge_upper_bound_b=1}
Let $k \geq 2$. Let $G$ be a graph on the vertex set $W$ given by Proposition \ref{l:planar_embedding_b=1}. 
 Let us assume that $G$ has at least $3k+1$ edges. Then
 \[|E(G)| \leq  \frac{12k}{4k+1}(|W| -2).\]
\end{lemma}
Having the lemma above, we can finish the proof of Theorem \ref{t:main_b=1}.

\smallskip
For $f_k(\A) = 0$ there is nothing to prove. From now on assume that $f_k(\A) \geq 1$, hence $f_{k-1}^{\ind}(\A) \geq k+1$.
Note that we can also freely assume that $G$ has at least $3k+1$ edges. 
Indeed, otherwise Theorem \ref{t:main_b=1} holds trivially: If  $|E(G)| \leq 3k$, then, by Observation \ref{o:simple_b=1}(\ref{o:simple_ii_b=1}), $f_k(\A) = |W| \le \frac{1}{k+1}f_1^{\ind}(\A)+\frac{3k}{k+1}$, which is subsumed in the desired bound. 

Assuming that $G$ has at least $3k+1$ edges, Lemma \ref{l:edge_upper_bound_b=1} combined with Observation \ref{o:simple_b=1}(\ref{o:simple_ii_b=1}) gives the desired bound:
\begin{equation*}
  f_k(\A) =  |W| \quad \leq \quad  \frac{4k+1}{4k^2-7k+1}f_{k-1}^{\ind}(\A).
\end{equation*}
Note that the fraction is strictly positive for any integer $k \geq 2$.

In order to finish the proof of Theorem \ref{t:main_b=1}, it remains  to prove Lemma \ref{l:edge_upper_bound_b=1}. We start with investigating the structure of triangles in $G$.

\heading{Structure of triangles in $G$.}
Let $G$ be a graph given by Proposition \ref{l:planar_embedding_b=1}.
To any triangle $a_{\sigma_1}a_{\sigma_2} a_{\sigma_3}$ in $G$ we assign a \emph{label} $\sigma_1 \cup \sigma_2 \cup \sigma_3$.
It turns out that each such label has the same cardinality, namely $k+2$ (Observation \ref{o:cyclic_structure_b=1}(\ref{o:i_b=1})). It follows from the construction of $G$ that all $\sigma_i$, $i \in [3]$, are distinct.
Denote by $H_\nu$ a subgraph of $G$ consisting of all triangles with the label $\nu$.
An important property is  that two subgraphs of distinct labels are edge-disjoint 
(Observation \ref{o:cyclic_structure_b=1}(\ref{o:iii_b=1})).

\begin{observation}\label{o:cyclic_structure_b=1}
\noindent
 \begin{enumerate}
 \item \label{o:i_b=1} If $a_{\sigma_1}, a_{\sigma_2}, a_{\sigma_3} \in V(G)$ span a triangle, then 
$|\sigma_1\cup \sigma_2 \cup \sigma_3| = k+2.$  
\item  \label{o:v_b=1} If two triangles share an edge, they have the same label.
  \item\label{o:ii_b=1} For a fixed $\nu$, $H_\nu$ has at most $k+2$ vertices and $3k$ edges. 
  \item\label{o:iii_b=1} For $\nu \neq \nu'$, $H_\nu$ and $H_{\nu'}$ are edge-disjoint. 
 \end{enumerate}
\end{observation}

\begin{proof}
\noindent
\begin{enumerate}
 \item For $i,j \in \{1,2,3\}$, we have 
  $|\sigma_i|=k+1$ and $|\sigma_i\cap \sigma_j| = k$ for $i \neq j$. Since $\sigma_i$'s are distinct, $|\sigma_1 \cup \sigma_2 \cup \sigma_3| \geq k+2$. By inclusion-exclusion principle, $|\sigma_1 \cap \sigma_2 \cap \sigma_3| \geq k-1$.
Moreover, $|\sigma_1 \cap \sigma_2 \cap \sigma_3|=k-1$ since $|\sigma_1 \cap \sigma_2 \cap \sigma_3|=k$ would imply that the tree
$G_{\sigma_1 \cap \sigma_2 \cap \sigma_3}$ contains a cycle $a_{\sigma_1}a_{\sigma_2}a_{\sigma_3}$. Hence,  $|\sigma_1 \cup \sigma_2 \cup \sigma_3| = k+2$.

\item By (\ref{o:i_b=1}), the label $\sigma_1 \cup \sigma_2 \cup \sigma_3$  of a triangle $a_{\sigma_1}a_{\sigma_2}a_{\sigma_3}$ in $G$ equals to the $(k+2)$-element set $\sigma_i \cup \sigma_j$, where $i, j \in [3], i \neq j$. The claim follows.

\item  By (\ref{o:i_b=1}), 
$|\nu|=k+2$, so $H_\nu$ is spanned by at most $k+2$ vertices $a_\sigma$, where $\sigma$ is a $(k+1)$-element subset of $\nu$. Since $H_\nu$ is a subgraph of a planar graph $G$, $H_\nu$ has at most $3(k+2)-6 = 3k$ edges.

\item It follows from (\ref{o:v_b=1}). Indeed, if $H_\nu, H_{\nu'}$ share an edge, then $\nu=\nu'$.\qedhere
\end{enumerate}
\end{proof}

\heading{Planar graphs.} Let $v(G),e(G),f(G)$ denote the number of vertices, edges and faces of a simple planar graph $G$, respectively. If the graph is clear from the context, we will use just $v,e,f$.
By Euler's formula, $v - e + f = c + 1$, where $c$ is the number of connected components of $G$.
Since $c \geq 1$, we have
\begin{equation}
 v - e + f \geq 2. \label{ineq:euler}
\end{equation}
Let us fix an embedding of $G$ and let $t_i(G)$ denote the number of faces of $G$ (in this embedding) with $i$ edges along its boundary. (If $G$ contains a bridge, the image of the bridge-edge 
is counted twice for the  face it lies in.) 

We will need the following lemma:

\begin{lemma}\label{l:planar_edge_upper_bound}
 Let $t \geq 3$. Let $G$ be a planar graph with at least $t + 1$ edges. Let $H$ be a subgraph of $G$ induced by triangles in $G$. That is, $H$ contains all edges that are contained in some triangle of $G$. Assume that $H$ can be decomposed as follows:
 \[
  H = \cup H_i, \text{ where each triangle of $G$ is in some $H_i$, all $H_i$ are edge-disjoint, and $3 \leq e(H_i) \leq t$}. 
 \]
Then there exists a constant $c_t < 3$ depending on $t$ such that $e(G) \leq c_t (v(G)-2)$. Specifically, we may take $c_t = \frac{12t}{4t+3}$.
\end{lemma}

With Lemma \ref{l:planar_edge_upper_bound} in hand, the proof of Lemma \ref{l:edge_upper_bound_b=1} is immediate.
\begin{proof}[Proof of Lemma \ref{l:edge_upper_bound_b=1}]
Remind that $H_\nu$ denotes the subgraph of $G$ consisting of all triangles in $G$ with the label $\nu$.
By Observation \ref{o:cyclic_structure_b=1}(\ref{o:i_b=1}), $|\nu|=k+2$, so $ \cup_{\nu \colon |\nu|=k+2} H_\nu$ contains all edges that are contained in some triangle of $G$. Moreover, since each triangle in $G$ has some label, it is contained in some $H_\nu$. Note that by Observation \ref{o:cyclic_structure_b=1}(\ref{o:ii_b=1}), we have $3 \leq e(H_\nu) \leq 3k$ for every nonempty $H_\nu$. By Observation~\ref{o:cyclic_structure_b=1}(\ref{o:iii_b=1}), $H_\nu$'s are edge-disjoint, therefore assumptions of Lemma~\ref{l:planar_edge_upper_bound} are satisfied for $t=3k$.
\end{proof}

To prove Lemma \ref{l:planar_edge_upper_bound}, we need the following folklore observation.

\begin{observation}\label{o:planar}
 Let $G$ be a simple planar graph with $v$ vertices, $e \geq 2$ edges and at most $T$ triangular faces. Then \[e \leq 2v-4+\frac{T}{2}.\]
 \end{observation} 
 
 \begin{proof}
 The proof is fairly standard and we include it just for completeness. For simplicity, we put $t_i=t_i(G)$.
Since $e \geq 2, t_2 = 0$ and we have
\begin{eqnarray}
 2e & = & 3t_3 + 4t_4 + 5t_5 + 6t_6 \cdots \nonumber \\
  & \geq & -t_3 + 4t_3 + 4t_4 + 4t_5 + 4t_6 \cdots \nonumber \\
  &= & -t_3 + 4f. \label{e:triangular_faces}
\end{eqnarray}
Using $t_3 \leq T$ and combining (\ref{e:triangular_faces}) and (\ref{ineq:euler}) the bound follows.\qedhere 
 \end{proof}

\begin{proof}[Proof of Lemma \ref{l:planar_edge_upper_bound}]
We fix an embedding of $G$. This embedding induces an embedding of $H$ and also of each $H_i$. We abuse the notation and speak about faces of $H_i$ while meaning faces in the drawing of $H_i$.
Since $e(H_i) \leq t < e(G)$, it follows that for every $i$, $G \setminus H_i$ contains at least one edge. Fix $i$. If all faces of $H_i$ are
triangular, then the remaining edges of $G$ are drawn into (at least) one of these triangular faces (in particular, such face cannot be a triangular face in the drawing of $G$). Hence, such $H_i$ contributes by at most $t_3(H_i)-1$ triangular faces to the number of triangular faces in the drawing of $G$. On the other hand, if $H_i$ contains a non-triangular face, it may happen that all the remaining edges of $G$ are drawn inside this face and such $H_i$ contributes by $t_3(H_i)$ triangular faces to $t_3(G)$. See Figure \ref{f:graphs} for illustration. Altogether, since by the assumption, each triangle of $G$ is in some $H_i$, we conclude that 
\begin{figure}
  \begin{center}
   \includegraphics[page=9]{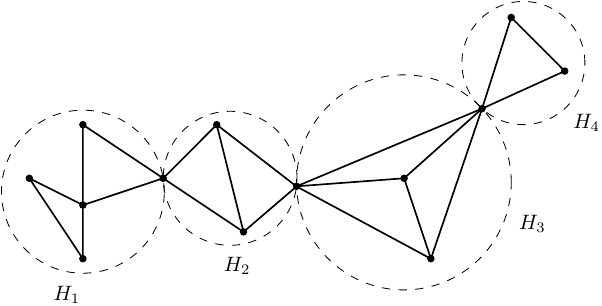}
   \caption{Both $H_1, H_2$ contain a non-triangular face while all faces of $H_3$ and $H_4$ are triangles.  Note that $t_3(H_3)=4$, but $H_3$ contributes only by three triangular faces to the drawing of $G$.}
   \label{f:graphs}
  \end{center}
 \end{figure}
\begin{equation*}
  t_3(G) \leq \sum_{H_i \text{ contains a non-triangular face}}t_3(H_i) + \sum_{\text{all faces of }H_i \text{ are triangles}}[t_3(H_i)-1] \label{e:decomposition}
\end{equation*}
We claim the following:
\begin{enumerate}
 \item[(i)] $H_i$ contains a non-triangular face $\quad \Rightarrow \quad t_3(H_i) \leq \frac23e(H_i)-\frac43$.
\item[(ii)] all faces of $H_i$ are triangles $\quad \Rightarrow \quad t_3(H_i) - 1  = \frac23e(H_i) - 1$.
 \end{enumerate}
 The claim (i) easily follows from the fact that there is at least one face in $H_i$ which is not a triangle. Since $e(H_i) \geq 3$, $t_2(H_i)=0$ and we have
\[ 
 2e(H_i) = 3t_3(H_i) + 4t_4(H_i) + \cdots \geq 3t_3(H_i) + 4.
\]
The claim (ii) easily follows from the fact that every face is a triangle, thus $2e(H_i)= 3t_3(H_i)$.

Since $\frac43 > 1$, we conclude that 
\[
 t_3(G) \leq  \sum_{i}\left[\frac23e(H_i)-1\right].
\]
For $e(H_i) \leq t$, we have 
\[
 \frac23e(H_i) - 1 \leq \frac{2t-3}{3t}e(H_i),
\]
hence 
\[
 t_3(G) \leq  \frac{2t-3}{3t}\sum_{i}e(H_i) \leq \frac{2t-3}{3t}e(G),
\]
where the last inequality follows from the fact that $H_i$'s are edge-disjoint subgraphs of $G$.
We have shown that the number of triangular faces in $\G$, the drawing of $G$, is at most  $\frac{2t-3}{3t}e(G)$, hence Observation \ref{o:planar}
gives the desired bound 
\[
 e(G) \leq \frac{12t}{4t+3}(v(G)-2).\qedhere
\] \qedhere
\end{proof}

\section{Proof of Theorem \ref{t:main_b}}\label{s:proof_t:main}
For $\sigma \subseteq [n]$ recall that $A_\sigma = \cap_{i \in \sigma} A_i$. Also recall that $N$ denote the nerve of $\A$. Let $\sigma \in N$ be a face of $N$ of size $k+1$. 
By assumption \ref{path_conn_k,k-1_b}, $A_\sigma$ has at most $b$ path-connected components.
We pick exactly $b$ distinct points $a_\sigma^i \in A_\sigma$ so that from each path-connected component of $A_\sigma$  we choose at least one point. Since $A_\sigma$ is open, we can indeed choose $b$ distinct points. 
Put
\[
 W_\sigma = \{a_\sigma^i \colon i=1,\ldots,b\} \qquad \text{and} \qquad
 W= \bigcup \{W_\sigma \colon \sigma \in N, |\sigma| = k+1\}.
\]
By assumption, no $k+2$ sets of $\A$ intersect, so $W_\sigma \cap W_{\sigma'} = \emptyset$ for $\sigma \neq \sigma'$.
By definition of $W_\sigma$, $ |W_\sigma| = b$ for any $\sigma \in N, |\sigma|=k+1$, hence 
  $|W| = bf_k(\A)$.
We define $V_\tau$ and $\Gamma$ analogically as before, that is,
 
\[
 V_\tau = \bigcup_{\sigma \in N \colon \sigma \supseteq \tau}W_\sigma \qquad \text{and} \qquad \Gamma = \{\tau \colon V_\tau \neq \emptyset\},
\]
where $\tau$ is a face of size $k$ which is contained in at least one $(k+1)$-element face of $N$.
Clearly, $W = \bigcup_{\tau \in N} V_\tau$ and $|\Gamma| = f^{\ind}_{k-1}(\A)$.

\heading{Embeddability.}
The aim is to construct a planar graph on the vertex set $W$  drawn properly inside $\bigcup \A$. We will proceed in several steps. 

First we construct graphs $G_\tau$ on the vertex sets $V_\tau \subseteq W$, together with their embeddings to $A_\tau$.
For each path-connected component of $A_\tau$ we look at all points $a_\sigma$ in this component for $\sigma \supset \tau$ and we form a tree on them. 
 The graph $G_\tau$ will be a forest formed by the union of all trees for all connected components of $A_\tau$ (see Figure \ref{f:crossings4}). 
Since $A_\tau$ is open, there is an embedding of the forest $G_\tau$ into $A_\tau$, we denote it by $\G_\tau \colon G_\tau \to A_\tau$. As before, we assume that all embeddings are piece-wise linear.

Note that $G_\tau$ has at most $b$ connected components, since $G_\tau$ restricted to each connected component of $A_\tau$ is a tree, and there are at most $b$ path-connected components of $A_\tau$. 

\begin{figure}
 \begin{center}
  \includegraphics[page=11]{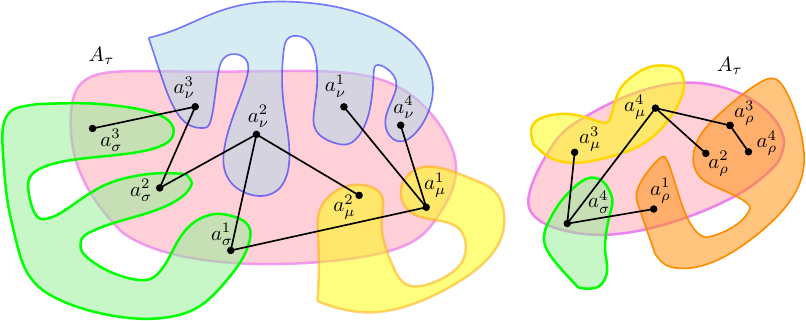}
  \caption{Construction of $G_\tau$ for $b=4$. Only two components of $A_\tau$ are shown.}
  \label{f:crossings4}
 \end{center}

\end{figure}

\begin{observation}\label{o:simple}
For any choice of forests $G_\tau$ and their embeddings $\G_\tau$ described above, the following is true. Define a graph $G$ on the vertex set $W$ as $G= \bigcup_{\tau \colon |\tau|=k} G_\tau$ and define a mapping 
 $\G \colon G \to \bigcup \A$ as $\G= \bigcup_{\tau \colon |\tau|=k}\G_\tau$.
 Then 
 \begin{enumerate}
  \item $|E(G)| \geq (k+1)|W| - bf^{\ind}_{k-1}(\A)$. In particular, $|E(G)| \geq (k+1)bf_k(\A) - b\binom{n}{k}$.\label{o:simple_ii}
  \item The embeddings $\G_\tau: G_\tau \to A_\tau$ can be modified so that
  under the mapping $\G$, the images of any two edges of $G$ intersect only finitely many times.  Moreover, for any edge $z$ of $G$, $\G(z) \cap \G(W)$ contains exactly the endpoints of $z$.
  \label{o:simple_iii}
 \end{enumerate}
\end{observation}

\begin{proof}
\begin{enumerate}

\item  
As before, each vertex of $G$ belongs to exactly $k+1$ graphs $G_\tau$. 
 We have already observed that  $G_\tau$ is a forest on $V_\tau$ with at most $b$ connected components, hence  $|E(G_\tau)| \geq |V_\tau|-b$. Recall $\Gamma = \{ \tau \colon V_\tau \neq \emptyset\}$. Thus

\[
 |E(G)|  = \sum_{\tau \in \Gamma} |E(G_\tau)| \geq \sum_{\tau \in \Gamma}{\left(|V_\tau|-b\right)} = (k+1)|W| - b|\Gamma|.
\]
Using that  $|W| = bf_k(\A)$ and $|\Gamma| = f^{\ind}_{k-1}(\A) \leq f_{k-1}(\A) \leq \binom{n}{k}$
, we get the desired bound. 

\item The proof is identical to the proof of Observation \ref{o:simple_b=1}(\ref{o:simple_iii_b=1}).\qedhere
\end{enumerate}
\end{proof}

\heading{Redrawing.} As before, the crucial part of the proof is to show that $G$ can be planar.  However, this time $G$ doesn't have to be a simple graph.
\begin{prop}\label{l:planar_embedding}
There is a choice of forests $G_{\tau}$ and their embeddings $\G_{\tau}$ satisfying the required conditions described above and such that the following holds. 
 The graph $G = \bigcup_{\tau \in \Gamma} G_\tau$ defined on the vertex set $W$ is a planar multigraph with $|E(G)| \leq 3 |W| + k(b-1)f_k(\A)$.
 Moreover, the union of embeddings $\G_\tau$ is an embedding of $G$ into $\bigcup \A$.
\end{prop}

\begin{proof}

The redrawing part is analogical to the redrawing procedure performed in the proof of Proposition \ref{l:planar_embedding_b=1}, the main difference is that here $G_\tau$ are forests and not trees. 
We will point out the differences in the proof.

Let $\mathcal G = \cup_{\tau} \mathcal G_\tau$ be a mapping, we describe how to get rid of all intersection points in $\Img(\G)$.  We note that all changes will be done on the level of forests $G_\tau$ and their embeddings $\mathcal G_\tau$.

 Let  $\G(e_\tau) \cap \G(e_\nu) \neq \emptyset$ for $e_\tau \in E(G_\tau), e_\nu\in E(G_\nu)$
  and the preimages of the \emph{intersection points} are not vertices of $e_\tau, e_\nu$.
  Let $x,y$ and $u,v$ be vertices of $e_\tau,e_\nu$, respectively, and let us denote their images under $\G$ the same way.
Fix an intersection point $p \subseteq \mathcal G(e_\tau) \cap \mathcal G(e_\nu)$, note that
$p \notin \{u,v,x,y\}$. 

Using condition \ref{no_k+2_intersect_b} of Theorem \ref{t:main_b}, it follows that $p \in A_\sigma$, where $\sigma =\tau \cup \nu$ and $|\sigma|=k+1$. 

 Let $B$ be the path-connected component of $A_\sigma$ containing $p$. By construction, there is at least one point from $W_\sigma$ contained in $B$. 
We choose one such point, say $a_\sigma^i$. Let $s$ be a path in $B$ between $p$ and $a_\sigma^i$.
The redrawing procedure now works exactly the same as in the case $b=1$. Important properties of this procedure are that the redrawing doesn't change the number of edges of $G_\tau$ and that $G_\tau$ remains acyclic. Repeating the procedure, we eliminate all intersection points in $\Img(\G)$. We conclude that $\mathcal G$ is an embedding and hence $G$ is a planar multigraph.

It remains to bound the number of edges.
Since $G$ is a union of forests, it contains no loops, however, it may contain multiple edges. We claim that multiplicity of each edge is at most $k+1$. Indeed, $G_\tau$ are simple graphs, so if there is a multiple edge in $G$ between vertices $a_\sigma^i, a_{\sigma'}^j$, $\sigma \neq {\sigma'}$, then both $a_\sigma^i, a_{\sigma'}^j$ belong to the forest $G_{\sigma \cap \sigma'}.$ However, for fixed $\sigma, \sigma'$ the forest is uniquely determined, hence there is just one edge between  $a_\sigma^i, a_{\sigma'}^j$. If $\sigma = \sigma'$, the vertex $a_\sigma^i, a_\sigma^j$, respectively, is contained in exactly $k+1$ graphs $G_\tau$, $\tau \subset \sigma$, where $|\sigma| = k+1$ and $|\tau|= k$. Hence, the multiplicity of the edge $a_\sigma^i, a_{\sigma'}^j$ is at most $k+1$. For illustration, see Figure \ref{f:multiple_edges}.

\begin{figure}
\begin{center}
  \includegraphics[page=12]{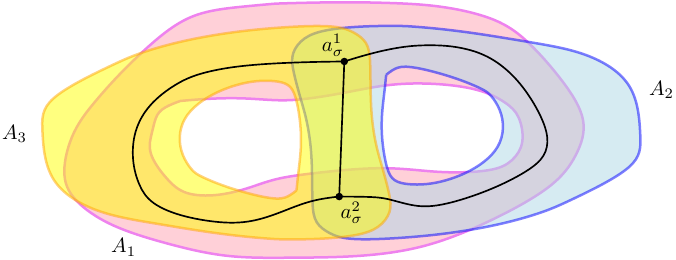}
  \caption{ Let $k=2$. There are three edges between points $a_\sigma^1, a_\sigma^2$, $\sigma=\{1,2,3\}$. These edges are edges of graphs $G_{1,2}, G_{1,3}$ and $G_{2,3}$, respectively.}\label{f:multiple_edges}
 \end{center}
 \end{figure}

To obtain the bound on the number of edges of $G$, we first consider a simple graph $\widetilde{G}$ on the vertex set $W$ defined as follows: $\widetilde{G}$ contains an edge between $u,v \in W$ if and only if there is an edge between $u$ and $v$ in $G$. 
In other words, $\widetilde G$ ignores multiple edges of $G$ and treats them as edges with multiplicity one. Since $\widetilde G$ is a simple subgraph of a planar graph, we get $|E(\widetilde G| \leq 3|W|$. It remains to bound how many edges of $G$ we have ignored.

We have seen that multiple edges in $G$ are just among the vertices of $W_\sigma$. Fix $\sigma$. Since $|W_\sigma|=b$, restricting a forest $G_\tau$, $\tau \subset \sigma$, to $W_\sigma$ gives at most $b-1$ edges, each of multiplicity at most $k+1$. Hence, each $\sigma$ contributes by at most $k(b-1)$ edges to $|E(G)| -|E(\widetilde G)|$. Note that the ``added'' multiplicity is $k$ instead of $k+1$, since one edge was already counted in $\widetilde G$.

Having exactly $f_k(\A)$ distinct $\sigma$, in total we get the desired bound
 
\[
 |E(G)| \leq |E(\widetilde G)| + k(b-1)f_k(\A) \leq 3|W| + k(b-1)f_k(\A). \qedhere
\]
\end{proof}

\heading{Proof of Theorem \ref{t:main_b}.} 
We just need to compare lower and upper bounds on the number of edges of $G$, given by Observation \ref{o:simple}(\ref{o:simple_ii}) and Proposition \ref{l:planar_embedding}, and use the fact that $|W|=bf_k(\A)$. For $k > 2b$ we get the following

\begin{eqnarray*}
 (k+1)|W| - bf^{\ind}_{k-1}(\A) \quad \leq |E(G)| \quad &\leq& \quad 3|W| +k(b-1)f_k(\A)\\
 f_k(\A) \quad &\leq& \frac{b}{k-2b}f_{k-1}^{\ind}(\A), 
\end{eqnarray*}
which finishes the proof.

 \section{Surface scenario}\label{s:manifold}

 First note that it is enough to prove the theorem for a connected surface. Indeed, otherwise for each connected component $M_i$ of $M$ we restrict the  family $\A$ to this component $M_i$.  Since every such restricted family $\A_i$ meets the conditions, we conclude that for each $i$, there exist constants $c_{1,i} > 0,c_{2,i} \geq 0$ depending only on $k,b$ and $M_i$ for which $f_k(\A_i) \leq c_{1,i}f_{k-1}^{\ind}(\A_i)+c_{2,i}$.
 Note that the number of connected components is finite since $M$ is compact.
 We have $f_k(\A) \leq \sum_i f_k(\A_i)$ and $f_{k-1}^{\ind}(\A_i) \leq f_{k-1}^{\ind}(\A)$, so putting $c_1= \sum_i c_{1,i}$ and $c_2=\sum_ic_{2,i}$ gives the desired result.
 
 It remains to prove Theorem \ref{t:manifold} for connected surfaces, so from now on assume that $M$ is connected. 
\heading{Graphs embedded into a connected surface.}
We start with a surface analogue of the folklore Euler's formula (\ref{ineq:euler}) for planar graphs $v-e+f \geq 2$.

\begin{observation}\label{o:euler_char}
Let $M$ be a compact connected surface with Euler characteristic $\chi$.
Let $G$ be a graph embedded into $M$ and suppose that $G$ has $f_0$ vertices, $f_1$ edges and $f_2$ faces in the corresponding embedding.
Then $f_0-f_1+f_2\geq \chi$. Consequently, a  graph embedded
into $M$ has at most $3(f_0 - \chi)$ edges.
\end{observation}
\begin{proof}
If the embedding triangulates the surface, we have equality.
If not, we perform the following procedure:

Subdivide each edge whose starting point and ending point coincide. This splits the edge into two and adds one
extra vertex, which means no change in $f_0-f_1 + f_2$.

Later on, we add edges to subdivide the non-cell face (i.e. face that is not a  disk).
This clearly adds edges, but it may or may not add faces. (Consider the case of three isolated points in the plane.)
But clearly, it does not change $f_0$ and  $f_2-f_1$ does not increase.

Let $G$ have the maximal number of edges graph embedded into $M$ can have. Then $2f_1 = 3f_2$ and the rest easily follows.
\end{proof}

The proof of Theorem \ref{t:manifold} goes along the same lines as the proofs of Theorems \ref{t:main_b=1} and \ref{t:main_b}. We construct the trees/forests $G_\tau$ exactly the same way as before (see Sections \ref{s:b=1} and \ref{s:proof_t:main}, respectively), the basic blocks---trees---can be PL-embedded into open sets in $M$. Observations \ref{o:simple_b=1}, \ref{o:simple} and Proposition \ref{l:planar_embedding_b=1} are independent of the surface, we only replace any occurrence of  ``planar graph'' by ``graph embedded into $M$''. The redrawing procedure works without any change since it is based only on the intersection patterns of the sets and the fact we have trees/forests. 
The surface $M$ starts playing a role when we want to bound the number of edges of the constructed graph $G$ embedded into $M$.

Let us start with the simpler case $b > 1, k > 2b$. The only change in Proposition \ref{l:planar_embedding} is that now $|E(G)| \geq 3(|W|-\chi) + k(b-1)f_k(\A)$, where we used that a graph $G$ on the vertex set $W$ embedded into $M$ has at most $3(|W|-\chi)$ edges (Observation \ref{o:euler_char}). 

Comparing the lower and upper bound on the number of edges of $G$ we get that
\[
 f_k(\A) \leq \frac{b}{k-2b}f_{k-1}^{\ind}(\A) - 3\chi.
\]
We put $c_1= \frac{b}{k-2b}$ and $c_2 = 0$ for $\chi \geq 0$ and $c_2=3\chi$ for $\chi < 0$. \medskip

To resolve the case $b=1, k\geq 2$ we need the following lemma.

\begin{lemma}\label{l:manifold_edge_upper_bound}
 Let $k \geq 2$ be an integer and let $G$ be a graph on the vertex set $W$ embedded into $M$ such that
 \begin{itemize}
  \item $G= \bigcup_{\tau\colon |\tau|=k} G_\tau$ is a simple graph, and $G_\tau$ are trees defined in Section \ref{s:b=1}.
  \item $|E(G)| \geq (k+1)|W|-f_{k-1}^{\ind}(\A)$.
 \end{itemize}
 Let us assume that $G$ has at least $t+1$ edges. Then 
 \begin{eqnarray*}
  |E(G)| \leq c_t(|W| -\chi), \text{ where } c_t= \frac{12t}{4t+3} \text{ and }
  t = 3(k+2-\chi).  
 \end{eqnarray*}
\end{lemma}

\begin{proof}
 In order to prove Lemma \ref{l:manifold_edge_upper_bound} we need Observation \ref{o:cyclic_structure_b=1} and Lemma \ref{l:planar_edge_upper_bound} and  for $G$ embedded into $M$.
Let us start with Lemma \ref{l:planar_edge_upper_bound} and inspect its proof. Replacing every usage of Euler's inequality (\ref{ineq:euler}) by Observation~\ref{o:euler_char}, we  conclude that $e(G) \leq c_t(v(G) - \chi)$, where $c_t = \frac{12t}{4t+3}$,  under the assumption that $G$ is embedded into $M$. To determine the value $t$, we need a surface version of Observation \ref{o:cyclic_structure_b=1}.
 
  The items (\ref{o:i_b=1}), (\ref{o:v_b=1}) and (\ref{o:iii_b=1}) in Observation \ref{o:cyclic_structure_b=1} are independent of the surface, so we focus on the item (\ref{o:ii_b=1}). As in Observation \ref{o:cyclic_structure_b=1}(\ref{o:ii_b=1}), $H_\nu$ has at most $k+2$ vertices, however, the number of edges differ. Since $H_\nu$ is a subgraph of a graph $G$ which is embedded into $M$, by Observation \ref{o:euler_char}, it has at most $3(k+2-\chi)$ edges. Hence, $t=3(k+2-\chi)$. Note that $t \geq 3k$, since Euler characteristic of a compact connected surface is at most two.

\end{proof}

Theorem \ref{t:manifold} now follows from the combination of the lower and upper bound on the number of edges, both stated in Lemma \ref{l:manifold_edge_upper_bound}.
For $|E(G)| \geq t+1$ we have
\begin{eqnarray*}
 (k+1)|W|-f_{k-1}^{\ind}(\A)  \leq \frac{12t}{4t+3}(|W|-\chi), \\
 \left(k-2+\frac{9}{4t+3}\right)|W| \leq f_{k-1}^{\ind}(\A)  - \frac{12t}{4t+3}\chi.
\end{eqnarray*}
Since $|W| = f_k(\A)$, there exist constants $c_1'>0, c_2'$ such that 
\[
 f_k(\A) \leq c_1' f_{k-1}^{\ind}(\A) + c_2',
\]
where both $c_1', c_2'$ depend on $t,b$ and $k$. 
Note that $c_2' < 0$ if and only if $\chi > 0$.

For $|E(G)| \leq t$, we directly have
\[
 (k+1)|W|-b^2f_{k-1}^{\ind}(\A)  \leq t,
\]
and again, there exist $c_1''=c_1''(b,t,k)>0$ and $c_2''=c_2''(b,t,k) > 0$ such that 
$
 f_k(\A) \leq c_1'' f_{k-1}^{\ind}(\A) + c_2''.
$
Defining $c_1:= c_1'+c_1''$, $c_2:= c_2' + c_2''$ for $\chi < 0$ and $c_2:=c_2''$ for $\chi \geq 0$ finishes the proof.

\section{Proof of Theorem \ref{t:construction}}\label{s:construction}

First we construct families $\A$ and $\mathcal B$ of \emph{closed} sets satisfying Theorem \ref{t:construction}.  Later on, we turn these families into families of open sets with the same nerve.
Our construction is inspired by the simple example we have seen in Section \ref{s:results}.

An \emph{arrangement of pseudolines} $\mathcal P$ is a collection of $n$ simple closed curves in the real projective plane $\R\mathbb{P}^2$, such that any two curves have exactly one point in common at which they cross.
Specifying a pseudoline $p_0 \in \mathcal P$ as the line at infinity, $\mathcal P$ induces the arrangement of $n-1$ pseudolines $\mathcal P \setminus \{p_0\}$ in $\R\mathbb{P}^2 \setminus \{p_0\}$. We will regard this arrangement as an arrangement in $\R^2$.
An arrangement $\mathcal P$ of pseudolines (in $\R^2$ or $\R\mathbb{P}^2$) is called \emph{simple} if no point belongs to more than two of these pseudolines. With an arrangement $\mathcal P$, there is an associated 2-dimensional cell-complex into which the pseudolines of $\mathcal P$ decompose $\R^2$ or $\R\mathbb{P}^2$, respectively. Let $p_3(\mathcal P)$ denote the number of (bounded) triangular faces (= triangles) among the cells of the complex associated with $\mathcal P$.

Already Gr\"unbaum \cite{grunbaum} observed that for any simple arrangement $\mathcal P$ of $n$ pseudolines in $\R \mathbb{P}^2$, $n \geq 4$, $p_3(\mathcal P) \leq \frac13n(n-1)$, since  two triangles cannot share an edge. There are $n(n-1)$ edges and every triangle uses three of them, so his observation follows. Hence, in every simple arrangement there is a pseudoline $\ell$ incident to at most $n-1$ triangles.

For $n \geq 5$, F\"uredi and Pal\'asti \cite{Furedi}  constructed a simple arrangement $\L_n$ of $n$ lines in $\R\mathbb{P}^2$ with $p_3(\L_n) \geq \frac13n(n-3)$. 
Let $\L_n$ be such an arrangement and choose a line $\ell$ incident to at most $n-1$ triangles as the line at infinity. The remaining Euclidean arrangement of $n-1$ lines has at least $\frac13n(n-3)-(n-1) = \frac13(n^2-6n+3)$ triangles. For $n \geq 6$, this value is strictly positive. Denote this arrangement by $\L$ and note that every two lines from $\L$ intersect.
 Let $ \{T_1,\ldots,T_m\}$ be the set of all triangular faces of the cell-complex associated with $\L$.  Denote by $b_i$ the barycenter of $T_i$ and for every $i$, consider a stellar subdivision of $T_i$ from $b_i$ (see Figure \ref{f_stellar}).
 Let us denote by $\L'$ the corresponding 2-dimensional cell complex.
 
 Then $\A = \{A_1,\ldots,A_n\}$
 is formed as follows (see Figure \ref{f:construction}):
 \begin{itemize}
 
  \item $A_1$  is the union of all lines $\ell_1, \ldots, \ell_{n-1}$ from $\L$,
  \item $A_i$ is a single line $\ell_i$ together with all triangular faces of $\L'$ sharing a segment with $\ell_i$.
 \end{itemize}
 
 We show that $\A$ satisfies the assumptions.
 Clearly, we have $n$ sets and any nonempty intersection is connected. Observe that no four sets share a point.
 Indeed, for $i < j < k$, $A_i \cap A_j \cap A_k \neq \emptyset$ implies that either $i=1$ and the intersection is the point 
 $\ell_j \cap \ell_k$,  or $A_i\cap A_j \cap A_k$ is the
 barycenter of the triangle formed by lines $\ell_i, \ell_j,\ell_k$.
 In both cases the intersection point is not contained in any other set from $\A$
 (in the first case since no three lines are concurrent, in the second case since the barycenter is contained in exactly three sets).
 
 From the previous considerations it also follows that the number of intersecting triples is equal to the number of line intersections plus the number of triangles in $\L$.
 
 By construction (as shown above), $p_3(\A) \geq \frac13(n^2-6n+3)$.
 Putting together, we have
 \[
  f_2(\A) \geq \binom{n-1}{2} +\frac13(n^2-6n+3),
 \]
which finishes the construction of a family of closed sets.

By considering the same type of construction for an arrangement of $n-1$ pseudolines in $\R^2$, we obtain a better bound. Roudneff \cite{Roudneff} showed that for infinitely many $n$, there is a simple arrangement $\mathcal P_n$ of $n$ pseudolines in $\R \mathbb P^2$ for which $p_3(\mathcal P_n) = \frac13n(n-1)$. As mentioned above, in any simple arrangement of $n$ pseudolines there is a pseudoline incident to at most $n-1$ triangles (in fact, in this arrangement every pseudoline is incident to exactly $n-1$ triangles). Putting arbitrary line as the line in infinity, we get that the remaining Euclidean arrangement contains at exactly $\frac13(n-1)(n-3)$ triangles. Using this arrangement of $n-1$ pseudolines as a starting point, we can repeat the construction above and obtain a family $\mathcal B$ of $n$ sets in $\R^2$ with $f_3(\mathcal B)=0$ and $f_2(\mathcal B) = \binom{n-1}{2} +\frac13(n-1)(n-3).$
\begin{figure}
 \begin{center}
   \includegraphics[page=3]{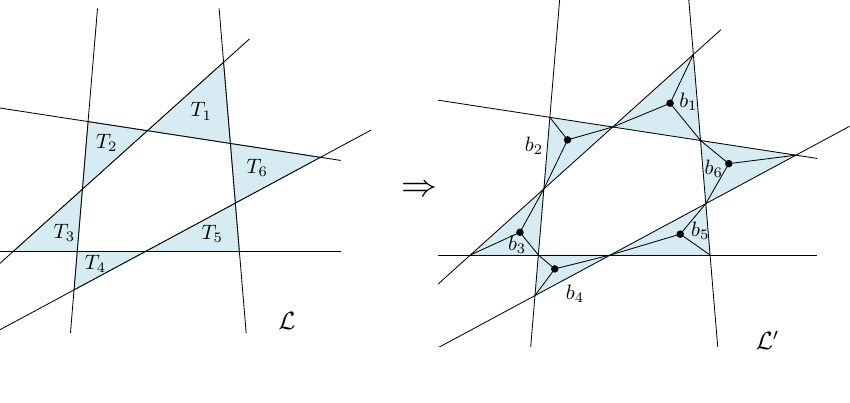}
   \caption{$\L$ and  $\L'$}\label{f_stellar}
 \end{center}
 \end{figure}
 
 \begin{figure}
 \begin{center}
   \includegraphics[page=4]{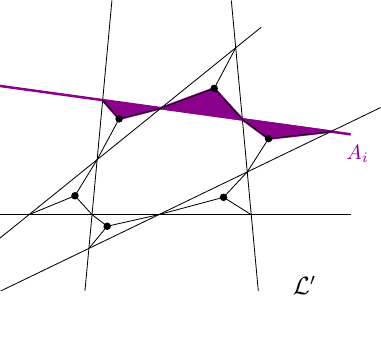}
   \caption{$\L'$ with highlighted set $A_i, i \geq 2.$}\label{f:construction}
 \end{center}
 \end{figure}

It is not difficult to make the sets of $\A$ and $\B$ open. Indeed, consider their open $\varepsilon$-neighborhood.
For $\varepsilon$ small enough, the nerve of the described family remains the same.

\section{Acknowledgment}
We are very grateful to Pavel Pat\'ak for many helpful discussions and remarks.
We also thank the referees for helpful comments, which greatly improved the presentation.

\bibliographystyle{alpha}
\bibliography{geom}
\end{document}